\documentclass[reqno,a4paper, 11pt]{amsart}
\usepackage[english]{babel}
\usepackage{amsmath}
\usepackage{amsfonts}
\usepackage{latexsym}
\usepackage{amsthm}
\usepackage{amssymb}

\def\D{\Bbb D}
\newcommand{\DU}{\overline{D}}

\def\s{\sigma}

\def\N{\Bbb N}

\def\om{\omega}

\def\e{\varepsilon}
\def\r{\rho}

\def\a{\alpha}

\newtheorem{theorem}{Theorem}
\newtheorem{lemma}[theorem]{Lemma}
\newtheorem{proposition}[theorem]{Proposition}
\newtheorem{corollary}[theorem]{Corollary}
\newtheorem{lettertheorem}{Theorem}
\newtheorem{letterlemma}[lettertheorem]{Lemma}

\theoremstyle{definition}

\newtheorem{example}[theorem]{Example}

\theoremstyle{remark}

\numberwithin{equation}{section}

\begin{document}
\title{Generalization of proximate order and applications}


\author{Igor Chyzhykov}
\address{School of Mathematics Science, Guizhou  Normal University, \newline
	Guiyang, Guizhou, 550001, China
}
\email{chyzhykov@yahoo.com}

\author{Petro Filevych}
\address{Department of Computational Mathematics and Programming, Lviv Polytechnic National University, Lviv, Ukraine}
\email{p.v.ﬁlevych@gmail.com}

\author{Jouni R\"atty\"a}
\address{Department of Physics and Mathematics, University of Eastern Finland, P.O. Box 111, FI-80101 Joensuu, Finland}
\email{jouni.rattya@uef.fi}



\date{\today}


\begin{abstract}
We introduce a concept of a quasi proximate order which is a generalization of a proximate order and allows us
to study efficiently analytic functions whose order and lower order of growth are different. We prove an existence theorem of a quasi proximate order, i.e. a counterpart of Valiron's theorem for a proximate order. As applications, we generalize and complement some results of M. Cartwright and C.~N.~Lin\-den on asymptotic behavior of analytic functions in the unit disc.
\end{abstract}

\maketitle

\section{Introduction and main results}

This paper concerns the asymptotic behavior of analytic functions $f$ whose order of growth $\sigma_M(f)$ and the lower order $\lambda_M(f)$, defined via the maximum modulus function $M(r,f)=\max\{ |f(z)|:|z|=r \}$, are distinct. In particular, we search for asymptotic lower estimates for $\log|f|$ and its $L^p$-mean on the circle of radius $r$ centered at the origin. Lindel\"of proximate order $\rho(r)$ has been widely and effectively used in the frame of such problems~\cite{GO,JK,Levin,Linden1956}. It allows one to majorize $\log M(r,f)$ by the flexible function $V(r)=r^{\rho(r)}$, where $\rho(r)\to\rho_M(f)$, as $r$ approaches either to $\infty$ (the case of entire functions) or to $1$ (the case of functions analytic in the unit disc). Indeed, by Valiron's theorem \cite{GO,JK,Levin} for each entire function of finite order there exists a proximate order $\rho(r)$ such that $\log M(r,f)\le V(r)$ for all $r$, and $\log M(r_n,f)=V(r_n)$ for some sequence $(r_n)$ tending to $\infty$. Such a proximate order is called a proximate order of an entire function $f$. This concept plays an essential role in the theory of functions of completely regular growth \cite{Levin}. However, the defect of this approach is that it completely ignores the value of the lower order. For entire functions $f$ of finite lower order $\lambda_M(f)$ there is a notion of a lower proximate order $\lambda(r)$~\cite{GO,Levin}, which allows to majorize $\log M(r,f)$ by $r^{\lambda_M(f)+o(1)}$ on a sequence of values of $r$ tending to $\infty$. Unfortunately, such a conclusion is in many cases far from being satisfactory. This leads us to the problem of constructing a majorant $V(r)$ for $\log M(r,f)$ such that, on one hand, it keeps the information about both the order $\rho_M(f)$ and the lower order $\lambda_M(f)$ well enough, and, on the other hand, it is sufficiently flexible. In particular we require that it satisfies Caramata's condition $V(2r)=O(V(r))$ as $r\to\infty$. It turns out that this is impossible in the frame of proximate order and its known generalizations. We solve this problem by introducing the notion of a quasi proximate order. It allows us to complement and generalize some results of M.~Cartwright \cite{CartWr33} and C.~N.~Linden~\cite{Li_mp,Li56_conj}.

We proceed towards the statements of our findings via necessary definitions. Let $1\le r_0<\infty$. A function $\sigma:[r_0,\infty)\to(0,\infty)$ is called a {\it quasi proximate order} if
there exist two constants $0\le\lambda<\rho<\infty$ and an associated function $A^*=A^*_\sigma:[r_0,\infty)\to(0,\infty)$ such that
		\begin{itemize}
    \item [(1)] $\sigma \in C^1[r_0,\infty)$;
    \item [(2)] $\displaystyle\limsup_{t\to\infty}\sigma(t)=\rho$;
    \item [(3)]	$\displaystyle\liminf_{t\to\infty}\sigma(t)=\lambda$;
    \item [(4)] $\limsup_{t\to\infty}|\sigma'(t)|t\log t<\infty$;
		\item [(5)] $A^*$ is nondecreasing and $A^*(t)\le t^{\sigma(t)}\le(1+o(1))A^*(t)$, as $t\to\infty$.
		\end{itemize}
Even though $t^{\sigma(t)}$ is not necessarily monotone, it follows from (2) and (4) that $t^{\sigma(t)}\asymp(2t)^{\sigma(2t)}$, as $t\to\infty$. Namely, (2) and (4) yield $\s(t)\le 2\r$, if $t\ge t_0$, and
	$$
	\s(t)-\frac{C}{\log t}\le\s(2t)\le\s(t)+\frac{C}{\log t},\quad t\ge t_0,
	$$
for some constants $C=C(\sigma)>0$ and $t_0=t_0(\sigma)>1$. 
These inequalities imply
$t^{\sigma(t)}\asymp(2t)^{\sigma(2t)}$ which together  with (5) yields  $A^*(2t)\lesssim A^*(t)$, as $t\to\infty$.

A function satisfying properties (1)--(3) and
	\begin{itemize}
    \item [(4')] $\limsup_{t\to\infty}|\sigma'(t)|t\log t=0$,
    \end{itemize}
is called {\it oscillating} or a {\it generalized proximate order}. If, in addition, $m<\lambda\le\rho<m+1$ for some $m\in\N\cup\{0\}$, then $\s$ is called a {\it Boutroux proximate order}. 
Such modifications of a proximate order, their properties and applications can be found in \cite[Chap.~2, \S5]{GO}, \cite{Chernolyas}, \cite{Boichuk}, \cite{MalKozSad}.
The notion of an oscillating proximate order was used to generalize the theory of functions of completely regular growth to classes of functions of non-regular growth~\cite{Boichuk,Chernolyas}. The main disadvantage of this approach is the lack of an existence theorem. All aforementioned generalizations can be used provided that   the upper limit (2) or the lower limit (3)  is finite. The reference \cite{Drasin74} deals with the case when it is not the case.
Finally, if $\rho=\lambda$ then a generalized proximate order coincides with a proximate order.

Our first main result shows that for each non-decreasing function $A$ satisfying $t^\alpha\lesssim A(t)\lesssim t^\beta$ for some $0\le\alpha<\beta<\infty$, there exists a quasi proximate order $\sigma$ such that the function $t^{\sigma(t)}$ is in a sense a smooth pointwise majorant of $A$ and still reflects the behavior of $A$ in a useful and natural way. The precise statement reads as follows.

\begin{theorem}\label{t:valir}
Let $1\le r_0<\infty$ and let $A$ be a positive continuous non-decreasing function on $[r_0,\infty)$ such that
		\begin{equation}\label{Eq:lim-sup-inf-conditions in Thm 1}
		\limsup_{t\to\infty}\frac{\log A(t)}{\log t}=\rho,
		\quad\liminf_{t\to\infty}\frac{\log A(t)}{\log t}=\lambda,
		\quad 0\le\lambda <\rho<\infty.
		\end{equation}
Then, for each fixed $\eta\in(0,\rho-\lambda)$, there exists a quasi proximate order $\sigma=\sigma_{\r,\lambda,\eta}:[r_0,\infty)\to(0,\infty)$, with the associated function $A^*=A^*_\sigma:[r_0,\infty)\to(0,\infty)$, such that
		\begin{itemize}
    \item[\rm(1)] $\displaystyle\limsup_{t\to\infty}\sigma(t)=\rho$;
    \item[\rm(2)] $\displaystyle\liminf_{t\to\infty}\sigma(t)=\lambda+\eta$;
    \item[\rm(3)] $A(t)\le t^{\sigma(t)}$ for all $r_0\le t<\infty$.
		\end{itemize}
\end{theorem}

It is tempting to think that in the statement (2) of Theorem~\ref{t:valir} one may replace $\lambda+\eta$ by $\lambda$. However, this is not possible without violating the condition $\limsup_{t\to\infty}|\sigma'(t)|t\log t<\infty$ of a quasi proximate order. In Section~\ref{Sec:example} we construct a non-decreasing function $A$ which shows this last claim. It also follows from the argument used that no generalized proximate order satisfying the limit conditions (2) and (3) of the definition exists for the constructed $A$. The proof of Theorem~\ref{t:valir} is also given in Section~\ref{Sec:example}. The proof itself is a technical handmade construction which relies on the use of suitable auxiliary functions induced by $A$ and inductive defining of $\sigma$ and $A^*$.   

Our primary interest in this paper relies on analytic functions in the unit disc $\D=\{z:|z|<1\}$. For this aim we first pull Theorem~\ref{t:valir} to the setting of unit interval $[0,1)$ by the substitution $t=\frac1{1-r}$, and then apply it to the logarithm of a suitable maximum modulus function. To be precise, recall that for an analytic function $f$ in $\D$, the order and the lower order are defined by
	$$
	\sigma_M (f) = \limsup_{r\to 1^-} \, \frac{\log^+ \log^+ M(r,f)}{\log{\frac{1}{1-r}}}
	$$
and
	$$
	\lambda_M(f) = \liminf_{r\to 1^-} \, \frac{\log^+ \log^+ M(r,f)}{\log{\frac{1}{1-r}}},
	$$
respectively. Let now $f$ be an analytic function in $\D$ such that $0\le\lambda_M(f)<\sigma_M(f)<\infty$. Write $A(t)=\log M(1-\frac 1t, f)$ for all $1\le t<\infty$. Then, for given $\eta\in(0,\sigma_M(f)-\lambda_M(f))$, there exist $\lambda$ and its associated function $A^*=A^*_\lambda$ on $[0,1)$ such that
\begin{itemize}
	\item [(1)] $\lambda \in C^{1}[0,1)$;
	\item [(2)] $\displaystyle\limsup_{r\to1^-} \lambda(t)=\sigma_M(f)$;
	\item [(3)]	$\displaystyle\liminf_{r\to1^-} \lambda(t)=\lambda_M(f)+\eta$;
	\item [(4)] $\displaystyle\limsup_{r\to1^-}|\lambda'(r)|(1-r)\log\frac1{1-r}<\infty$;
	\item [(5)] $A^*(r)\le \frac{1}{(1-r)^{\lambda(r)}} \le (1+o(1))A^*(r)$, as $r\to1^-$;
	\item [(6)] $A^*$ is nondecreasing and $A^*\left(\frac{1+r}{2}\right)\lesssim A^*(r)$ for all $0\le r<1$;
	\item [(7)] $\log M(r,f)\le\frac{1}{(1-r)^{\lambda(r)}}$ for all $0\le r<1$.
\end{itemize}
To see this, it is enough to set $\lambda(r)=\sigma(\frac{1}{1-r})$, where $\sigma$ is that of Theorem~\ref{t:valir} applied to $A$. Then (1)--(6) are immediate while (7) comes from the statement (3) of the said theorem. If the properties (1)--(7) hold, then $\lambda$ is called a {\it quasi proximate order of the analytic function} $f$ in~$\D$, related to $\eta\in(0,\sigma_M(f)-\lambda_M(f))$. Further, if 
\begin{itemize}
	\item [(4')] $\displaystyle\limsup_{r\to1^-}|\lambda'(r)|(1-r)\log\frac1{1-r}=0$,
\end{itemize}
then $\lambda$ is a {\it generalized proximate order} of $f$.
 
We now proceed towards the statement of our second main result. Some more definitions are needed. Let 
	\begin{gather*}
	m_p(r,\log|f|)=\left( \frac 1{2\pi} \int_0^{2\pi} |\log|f(re^{i\theta})||^p\, d\theta\right)^{\frac 1p}, \quad 0<r<1,
	\end{gather*}
and, by following~\cite{Li_mp}, define
	\begin{gather*}
	\rho_p(f)=\limsup _{r\to1^-}\frac{\log^+ m_p(r,\log|f|)}{-\log(1-r)}, \quad 
	\lambda_p(f)=\liminf _{r\to1^-}\frac{\log^+ m_p(r,\log|f|)}{-\log(1-r)}.
	\end{gather*}
It is known \cite{Li_mp} that $\rho_p(f)$ is an increasing function of $p$ and $p\rho_p(f)$ is convex on $(0,\infty)$.
Further, define $\rho_\infty(f)=\lim_{p\to\infty}\rho_p(f)$ as in~\cite{rho_infty}. Linden~\cite{Li_mp} proved the identity $\sigma_M(f)=\rho_\infty(f)$ and showed that 	
	\begin{equation}\label{e:linden_ineq}
	\rho_p(f)\le \sigma_M(f)\le \rho_p(f)+\frac 1p, \quad  0\le p<\infty, 
	\end{equation}
provided $\sigma_M(f)\ge 1$. In general, $\sigma_M(f)\le\rho_\infty(f)$ for all analytic functions in $\D$. Observe that the left-hand inequality in \eqref{e:linden_ineq} is no longer true if $\sigma_M(f)<1$. A Blaschke product~$B$ such that $\rho_p(B)=1-\frac 1p$ while $\sigma_M(B)=0$ gives a counter example~\cite{ChIJM}. Further properties of the order $\rho_\infty(f)$ can be found in \cite{rho_infty}.

Our aim is to establish a counterpart of \eqref{e:linden_ineq} for $\lambda_p(f)$ and $\lambda_M(f)$. This is what we obtain from our next result. Recall that the upper density of a measurable set $E\subset[0,1)$ is defined as
$\DU(E)=\limsup_{r\to1^-}\frac{m(E\cap[r,1))}{1-r}$, where $m(F)$ denotes the Lebesgue measure of the set $F$.

\begin{theorem} \label{t:p_lower_order}
Let $f$ be an analytic function in $\D$ such that it admits a generalized proximate order $\lambda$, and either $1<\lambda_M(f)<\sigma_M(f)<\infty$ or $0\le\lambda_M(f)<\sigma_M(f)\le 1$. Let $1\le p<\infty$ and $\varepsilon>0$. Then
	\begin{equation}\label{e:mp_final-countdown}
	m_p(r,\log|f|)\lesssim\frac{1}{(1-r)^{\lambda(r)\left(1-\frac1{\sigma_M(f)}\right)^++ 1+\varepsilon}}, \quad r\to 1^-.
	\end{equation}
In particular, if $\eta\in(0,\sigma_M(f)-\lambda_M(f))$ then 
	\begin{equation}\label{e:mp_final-countdown2}
	m_p(r,\log|f|)\lesssim\frac{1}{(1-r)^{\left(\lambda_M(f)+\eta\right)\left(1-\frac1{\sigma_M(f)}\right)^++ 1+\varepsilon}},\quad r\in E,
	\end{equation}
where $E=E_{\varepsilon,\eta}\subset[0,1)$ is of upper density one.
\end{theorem}

The method we employ in the proof does not allow us to treat the case $\lambda_M(f)\le 1<\sigma_M(f)$ and therefore it remains unsettled. 

\begin{corollary} \label{c:m_p}
Let $f$ be an analytic function in $\D$ such that it admits a generalized proximate order $\lambda$, and either $1<\lambda_M(f)<\sigma_M(f)<\infty$ or $0\le\lambda_M(f)<\sigma_M(f)\le 1$. Let $\eta\in(0,\sigma_M(f)-\lambda_M(f))$ and  $1\le p<\infty$. Then 
	\begin{equation}\label{pilili}
	\lambda_p(f)\le\lambda_M(f)\left(1-\frac1{\sigma_M(f)}\right)^++1, \quad \lambda_M(f)\le \lambda_p(f)+\frac 1p.
	\end{equation}
\end{corollary}

Corollary~\ref{c:m_p} complements \eqref{e:linden_ineq}. Indeed, the left-hand inequality in \eqref{e:linden_ineq} is the limit case $\lambda_M(f)=\sigma_M(f)$ of the first inequality in \eqref{pilili} with $\sigma_M(f)\ge1$. However, the sharpness of Corollary~\ref{c:m_p} itself remains unsettled. We just note here that a similar phenomenon appears in estimates of growth for the central index of analytic functions in the unit disc in terms of orders of their maximum term~\cite{Sons68}.
	
The proof of Theorem~\ref{t:p_lower_order} occupies most of the body of the remaining part of the paper. It follows the scheme of the proof of \cite[Theorem~1]{Li_mp} to some extent but a substantial amount of different arguments is needed. Some of the auxiliary results obtained on the way to the proof are of independent interest. The first step towards to proof of the theorem is to estimate the number of zeros of an analytic function in polar rectangles in terms of its proximate order. Our result in this direction is Proposition~\ref{p:i_est} which is proved in Section~\ref{sec:polar-rectangle}. The second step is to establish an appropriate generalization of the following lemma due to Cartwright~\cite[Lemma~1]{CartWr33}.

\begin{letterlemma}\label{Lemma:B} Let $0<R<\infty$, $0<\alpha<\infty$ and $\frac{\pi}{2\alpha}<\beta<\infty$. Let $F$ be analytic and non-vanishing in the truncated sector $\Omega=\Omega_{R,\beta}=\{re^{i\theta}:R\le r<\infty,\,|\theta|\le \beta\}$ such that
	\begin{equation}\label{e:cart_assum}
	\log|F(z)|<B|z|^\alpha,\quad z\in\Omega,
	\end{equation}
	for some constant $B>0$. Then, for given $\delta>0$, there exists a constant $K=K(\delta)>0$ such that
	$$
	\log|F(re^{i\theta})|>-KBr^\alpha, \quad |\theta|\le \beta-\delta,\quad R\le r<\infty.
	$$
\end{letterlemma}

Lemma~\ref{Lemma:B} plays a crucial role in the proof of \cite[Theorem~2]{CartWr33}. Linden~\cite{Li56_conj} generalized Cartwright's lemma to the case in which the power function $|z|^\alpha$ is replaced by $V(|z|)=|z|^{\rho(|z|)}$, where $\rho(r)$ is a proximate order of $F$, and applied it to analytic functions in the unit disc. In \cite{NikNk74} and \cite{Bor_min} the authors considered the question of when the conclusions of the lemma can be strengthened. They showed that indeed a stronger conclusion can be made if certain extra hypothesis on $V$ is imposed.

In this paper we are interested in the question of {\it to what extent the hypotheses of Lemma  \ref{Lemma:B} can be relaxed and still have the same lower estimate?} It appears that it is enough to assume that $\rho(r)$ is a generalized proximate order. It is worth underlining here that this allows $f$ to be of non-regular growth. Moreover, it seems that the conclusion is no longer true for a quasi proximate order. Our generalization of Cartwright's lemma reads as follows.
 
\begin{proposition}\label{cor:cartright} Let $l:[1,\infty)\to(0,\infty)$ be a generalized  proximate order such that $0<\liminf_{t\to\infty}l(t)=l_1<\limsup_{t\to\infty}l(t) =l_2<\infty$. Let $\varepsilon>0$, $0<q<1$ and $0<R<\infty$. Let $G$ be analytic and non-vanishing such that 
 	\begin{equation*}
 	\log|G(re^{i\theta})|<r^{\frac{l(r)}{1+\e}}
 	\end{equation*}
 	on the domain $\left\{re^{i\theta}:R<r<\infty,\,|\theta|\le\frac{\pi}{2l(r)q}\right\}$. Then, for each $0<\delta<\pi/2l_2q$, we have
 	\begin{equation*}
 	\log|G(re^{i\theta})|
 	>-r^{l(r)}, \quad |\theta|\le\frac{\pi}{2l(r)q}-\delta,\quad r\to \infty.
 	\end{equation*}
 \end{proposition}

The proposition allows us to deduce the following result concerning real parts of analytic functions in the unit disc. It is a counterpart of \cite[Theorem 3]{Li56_conj} and of independent interest.
 
\begin{theorem}\label{t:cartright}
	Let $\sigma$ be a generalized proximate order such that $\lambda(r)=\sigma(1/(1-r))$ satisfies $1<\liminf_{r\to1^-}\lambda(r)<\limsup_{r\to1^-}\lambda(r)<\infty$. Let $\varepsilon>0$ and let $f$ be analytic in $\D$ such that $f(0)=0$ and
	\begin{equation}\label{Eg:hypothesis-cartright}
	\Re f(re^{i\theta})<\frac{1}{(1-r)^{\lambda(r)}},\quad 0\le \theta<2\pi, \quad r_0<r<1,
	\end{equation}
	for some $r_0\in(0,1)$. Then 
	$$
	|\Re f(re^{i\theta})|<\frac{1}{(1-r)^{(1+\varepsilon)\lambda(r)}},\quad 0\le \theta<2\pi, \quad r\to1^-.
	$$
\end{theorem}

Observe that the hypothesis $1<\liminf_{r\to1^-}\lambda(r)$ in Theorem~\ref{t:cartright} is necessary. Namely, if $\limsup_{r\to 1^-} \lambda(r)<1 $ one can only say that $|\Re f(re^{i\theta})|\lesssim\frac{1}{1-r}$ by~\cite{CartWr33}. This upper bound is the best possible as is seen by considering the function $f(z)=1-\frac{1}{(1-z)^\alpha}$ with $\alpha\in(0,1)$.

Proposition~\ref{cor:cartright} and Theorem~\ref{t:cartright} are established in Section~\ref{sec:cartright}. The proofs do not come free but here we only mention one specific tool used which is the Warschawski mapping theorem, stated as Theorem~\ref{t:war}. Finally, in Section~\ref{sec:final} we pull all the discussed things together, and prove Theorem~\ref{t:p_lower_order} and Corollary~\ref{c:m_p}. 

To this end, couple of words about the notation used throughout the paper. The letter $C=C(\cdot)$ will denote an absolute constant whose value depends on the parameters indicated
in the parenthesis, and may change from one occurrence to another.
We will use the notation $a\lesssim b$ if there exists a constant
$C=C(\cdot)>0$ such that $a\le Cb$, and $a\gtrsim b$ is understood
in an analogous manner. In particular, if $a\lesssim b$ and
$a\gtrsim b$, then we write $a\asymp b$ and say that $a$ and $b$ are comparable. Moreover, the notation $a(t)\sim b(t)$ means that the quotient $a(t)/b(t)$ approaches one as $t$ tends to its limit.

\section{Existence of quasi proximate order}\label{Sec:example}

In this section we prove Theorem~\ref{t:valir} and discuss the necessity of its hypotheses. We begin with the proof of the theorem.

\begin{proof}[Proof of Theorem \ref{t:valir}]
Define $d(t)=\frac{\log^+A(t)}{\log t}$ for all $t\in[\max\{r_0, e\},\infty)$. Let $(\e_n)_{n=1}^\infty$ be a strictly decreasing sequence of strictly positive numbers such that $\varepsilon_1<\min\{1,\eta\}$ and $\lim_{n\to\infty}\varepsilon_{n}=0$. The continuity of $A$ together with \eqref{Eq:lim-sup-inf-conditions in Thm 1} allows us to pick up the infinite sequences $(r_n)_{n\in\N}$, $(r_n')_{n\in\N}$ and $(r_n^*)_{n\in\N}$ such that
	\begin{itemize}
  \item[\rm(i)] $r_n<r_n^*<r_n'<r_{n+1}$ for all $n\in\N$;
  \item[\rm(ii)] $d(r_n')=\lambda+\frac\eta2$ for all $n\in\N$;
  \item[\rm(iii)] $d(r_1)=\rho -\varepsilon_{1}$ and $r_{n+1}=\min\{ r\ge r_n': d(r)=\rho -\varepsilon_{n+1}\}$ for all $n\in\N$;
  \item[\rm(iv)] $(r_n^*)^{\lambda+\eta}=(r_n')^{\lambda+\frac\eta2}$ for all $n\in\N$.
	\end{itemize}
The existence of such sequences can be seen, for example, by arguing as follows. First, consider a sequence $(r_n')_{n\in\N}$ consisting of infinitely many points satisfying (ii). Then define $(r_n^*)_{n\in\N}$ by (iv), and by passing to suitable subsequences if necessary we have $r_n'<r_{n+1}^*<r_{n+1}'$. By defining $(r_n)_{n\in\N}$ by (iii) we now have $r_n'<r_{n+1}$, and by passing once more to subsequences we obtain (i). Observe that the value of $d$ at the points $r_n$ and $r_n'$ is known precisely by (ii) and (iii), while (i), (ii) and (iv) together with the monotonicity of $A$ yield the estimate
	\begin{equation}\label{Eq:estimate d(r_n^*)}
	d(r_n^*)=\frac{\log^+A(r_n^*)}{\log r_n^*}
	\le\frac{\lambda+\eta}{\lambda+\frac{\eta}{2}}\frac{\log^+A(r_n')}{\log r_n'}
	=\frac{\lambda+\eta}{\lambda+\frac{\eta}{2}}d(r_n')=\lambda+\eta,\quad n\in\N.
	\end{equation}

Let now $E_1=\{r\in[r_0, r_1^*]:d(r)\ge\rho-\varepsilon_1\}$ and
	$$
	E_n=\{r\in[r_n, r_n^*]: d(r)\ge \rho-\varepsilon_n\}, \quad n\in\N\setminus\{1\}.
	$$
Then each $E_n$ is non-empty by the property (iii). Write $M_n=\max\{d(r):r\in E_n\}=\max\{d(r):r\in[r_n, r_n^*]\}$, and let $R_n\in E_n$ such that $d(R_n)=M_n$. Then $M_n\to\rho$, as $n\to\infty$, and, by fixing $r_1$ sufficiently large, $\rho-\varepsilon_n\le M_n\le\rho+\varepsilon_1$ for all $n\in\N\setminus\{1\}$ by \eqref{Eq:lim-sup-inf-conditions in Thm 1}. Define
	\begin{equation}\label{Eq:(a)}
  \s(t)=\left\{
        \begin{array}{ll}
        M_1,&\quad t\in[r_0,R_1]\\
        M_n,& \quad t\in[r_n,R_n],\quad n\in\N\setminus\{1\},
        \end{array}\right.
  \end{equation}
and set
	$$
	A^*(t)=t^{\sigma(t)},\quad t\in[r_0,R_1]\cup\Bigl(\bigcup_{n\in\N\setminus\{1\}}[r_n,R_n]\Bigr).
	$$
Then $A^*$ is trivially nondecreasing and continuous on each $[r_n,R_n]$ and on $[r_0,R_1]$. Since $A$ is nondecreasing by the hypothesis, the definition of $d$ and the property (iv) imply
	$$
	A(t)\le A(r_n')=(r_n')^{\lambda+\frac\eta2}=(r_n^*)^{\lambda+\eta}\le t^{\lambda+\eta},\quad t\in [r_n^*, r_n'],\quad n\in\N,
	$$
and hence $d(t)\le\lambda+\eta$ for all $t\in [r_n^*,r_n']$ and $n\in\N$. This extends \eqref{Eq:estimate d(r_n^*)}. 

Let
	$$
	D(t)=\max \{ \max\{ d(r): t\le r\le r_n^*\}, \lambda+\eta\},\quad  t\in [R_n, r_n^*],\quad n\in\N\setminus\{1\}.
	$$
Then $D(t)$ is nonincreasing, and $D(r_n^*)=\max\{d(r_n^*),\lambda+\eta\}=\lambda+\eta$. Define
	$$
	A^*(t)=t^{D(t)},\quad t\in[R_n,r_n^*],\quad n\in\N\setminus\{1\}.
	$$
Then $A^*$ is continuous at each $R_n$, and nondecreasing on each $[R_n,r_n^*]$. Namely, $[R_n, r_n^*]=B\sqcup C$, where $B=\{t\in [R_n,r_n^*]:d(t)=D(t)\}$ is a closed set and $C$ is a union of open intervals $\{\Delta_{k,n}\}_{k}$. By the definition of $D$, we have $A^*(t)=e^{\log^+A(t)}$ on $B$, and $A^*(t)=t^{\lambda+\eta}$ on $C$. The monotonicity of $A^*$ on each $[R_n,r_n^*]$ now follows from that of $A$ and the definition of~$D$.

We next define $\sigma$ on $[R_n,r_n^*]$. In order to do so, denote $u_{0,n}=R_n$ and $u_{0,n}^*=\max\{t\in[R_n,r_n^*]:D(t)=D(R_n)\}$, and define
	$$
	t_{1,n}=\min\{ t\in \mathbb{N}: t\ge u_{0,n}+1, D(t)<D(u_{0,n})=D(R_n)\}.
	$$
Then obviously $t_{1,n}\le u_{0,n}^*+2$. We now define
	$$
	\sigma(t)=\sigma(u_{0,
 n})=M_n,\quad R_n=u_{0,n}\le t\le t_{1,n}.
	$$
Let
	$$
	u_{1,n}=\min\{ t> t_{1,n}:D(t)=y_{1,n}(t)\}
	$$
be the abscissa of the point in which the graphs $y=D(t)$ and
	$$
	y=y_{1,n}(t)=M_n-(\rho+1)\log\log t+(\rho+1)\log\log t_{1,n}
	$$
intersect. Note that $t^{y_{1,n}(t)}$ is decreasing on $[t_{1,n},r_n^*]$, because
	\begin{equation*}
	\begin{split}
	(y_{n,1}(t)\log t)'
	&=\frac 1t\left(-(\rho+1)+M_n-(\rho+1)(\log\log t-\log\log t_{1,n})\right)\\
	&<\frac 1t (-\rho-1+M_n)<0.
	\end{split}
	\end{equation*}
Since $D(t)\log t=\log A^*(t)$, and $A^*$ is nondecreasing and unbounded, the point $u_{1,n}$ exists and is unique. We set
  \begin{equation}\label{Eq:(d)}
	\sigma(t)= \sigma (t_{1,n})-(\rho+1) \log\log t+ (\rho+1) \log\log t_{1,n}, \quad t_{1,n} \le t\le u_{1,n}.
	\end{equation}
Let $u_{1,n}^*=\max \{t\in[u_{1,n},r_n^*]:D(t)=D(u_{1,n})\}$, and then choose
  $$
	t_{2,n}=\min\{t\in \mathbb{N}: t\ge u_{1,n}+1, D(t)<D(u_{1,n})\}.
	$$
Continue the process as above until the function $\s$ is defined on the whole interval $[R_n, r_n^*]$. Since $u_{k+1, n}\ge t_{k+1,n}\ge u_{k,n}+1$, the process finishes after a finite number of steps.

We will show next that
	\begin{equation}\label{e:a_astar_com}
	A^*(t)\le t^{\sigma(t)} \le (1+o(1))A^*(t), \quad t\in [R_n, r_n^*],\quad n\to\infty.
	\end{equation}
The left-hand inequality follows immediately from the construction. Moreover, $A^*(t)=t^{D(t)}=t^{\sigma(t)}$ for each $t\in[u_{k,n}, u_{k,n}^*]$. Note that $t_{k+1,n}\le u_{k,n}^*+2$ for each $k$ by the definition of $t_{k+1,n}$. Further, for $t\in[u_{k,n}^*,t_{k+1,n}]$ we have
	\begin{equation*}
	\begin{split}
  t^{\sigma(t)}
	&=t^{D(u_{k,n}^*)}
	\le (u_{k,n}^*+2)^{D(u_{k,n}^*)}
	=(u_{k,n}^*)^{D(u_{k,n}^*)}\left(1 +\frac 2 {u_{k,n}^*}\right)^{D(u_{k,n}^*)}\\
  &
	\le A^*(t)9^{\frac{D(u_{k,n}^*)}{u_{k,n}^*}}\le A^*(t)9^{\frac{\r+1}{R_n}}
	\le A^*(t)\left(1+\frac{9^{\r+1}}{R_n}\right),\quad n\in\N\setminus\{1\}.
	\end{split}
	\end{equation*}
Since $t^{\sigma(t)}$ is decreasing for $t\in [t_{k+1,n}, u_{k+1,n}]$ by its definition, we obtain
	$$
	t^{\sigma(t)}
	\le t_{k+1,n}^{\sigma(t_{k+1,n})}
	\le A^*(t_{k+1,n})\left(1+\frac{9^{\r+1}}{R_n}\right)
	\le A^*(t)\left(1+\frac{9^{\r+1}}{R_n}\right).
	$$
Therefore \eqref{e:a_astar_com} is proved.

It follows from \eqref{e:a_astar_com} that
	$$
	(r_n^*)^{\lambda+\eta}=A^*(r_n^*)\le (r_n^*)^{\sigma(r_n^*)} =(1+\delta_n) (r_n^*)^{\lambda+\eta},
	$$
where $0\le\delta_n\le\frac{9^{\r+1}}{R_n}$. Hence
	\begin{equation}\label{Eq:(b)}
	\lambda+\eta
	\le\sigma(r_n^*)
	\le\lambda+\eta +\frac{9^{\rho+1}}{R_n\log R_n}.
	\end{equation}
We then define
\begin{equation}\label{e:sigma_new}
  \sigma(t)=\sigma(r_n^*)+ C_n(\log\log t-\log\log r_n^*), \quad r_n^*\le t\le r_n',
\end{equation}
where
	$$
	C_n=\frac{M_{n+1}-\sigma(r_n^*)}{\log\frac{\lambda+\eta}{\lambda+\frac \eta2}}
	\le \frac{\rho+\varepsilon_1-\lambda-\eta}{\log\frac{\lambda+\eta}{\lambda+\frac \eta2}}.
	$$
Since $(\lambda+\eta)\log r_n^*=(\lambda+\frac \eta2)\log r_n'$ by (iv), this yields $\sigma(r_n')=M_{n+1}$.
We then define
	\begin{equation}\label{Eq:(c)}
	\sigma(t)=M_{n+1},\quad r_n'\le t\le r_{n+1},
	\end{equation}
and $A^*(t)=t^{\sigma(t)}$ for all $t\in [r_n^*, r_{n+1}]$. Then we
make the next step.

The statement (1) follows from the construction, see, in particular, \eqref{Eq:(a)}, the definition of $M_n$ and the hypothesis \eqref{Eq:lim-sup-inf-conditions in Thm 1}. The statement (2) is a consequence of the construction and \eqref{Eq:(b)}.
Statement (3) is obvious on each $[r_n,R_n]$ by the definition of $M_n$. On $[R_n,r_n^*]$ we have
	$$
	A(t)\le t^{d(t)}\le t^{D(t)}=A^*(t)\le t^{\s(t)}
	$$
by \eqref{e:a_astar_com}. On the remaining interval $[r_n^*,r_{n+1}]$ we have (3) by the construction, see \eqref{e:sigma_new} and \eqref{Eq:(c)}, and the hypothesis on the monotonicity of $A$.

It remains to show that $\s$ is a proximate order and $A^*$ satisfies the properties of an associated function. We have already seen that $\s$ satisfies the property (2) and also (3), with $\lambda+\eta$ in place of $\lambda$. Moreover, (5) is satisfied by \eqref{e:a_astar_com} because $A^*(t)=t^{\s(t)}$ for all $t\in[r_0,\infty)\setminus\bigcup_{n\in\N\setminus\{1\}}[R_n,r_n^*]$. Furthermore, by \eqref{Eq:(d)} and \eqref{Eq:(c)}, we also have $|\sigma'(t)|\le\frac{K}{t\log t}$, where $K=\max\{\sup_n C_n, \rho+1\}$, except for countably many points, where the derivative does not exist. We can slightly modify the continuous function $\sigma$ in small neighborhoods of these points without changing the other properties to get (1) and (4). The associated function $A^*$ is continuous and increasing by the construction.
Theorem~\ref{t:valir} is now proved.
\end{proof}

The following example shows that in the statement of Theorem~\ref{t:valir} the parameter $\eta$ must be strictly positive. Observe also that no generalized proximate order satisfying the limit conditions (2) and (3) of the definition exists for the constructed $A$.

\begin{example}
Let $0<\lambda<\rho<\infty$, $r_1=2$ and $r_{n+1}=r_n ^{\frac{\rho}{\lambda}}$ for all $n\in\mathbb{N}$. Define
	$$
	A(r)=r_n^\rho,\quad r_n<r\le r_{n+1},\quad n\in\mathbb{N}.
	$$
	Then
	$$
	A(r_n^+)=r_n^\rho,\quad A(r_{n+1})=r_n^\r=r_{n+1}^\lambda, \quad n\in\N,
	$$
	and
	\begin{equation}\label{Eq:example}
	r^\lambda\le A(r)<r^\rho,\quad r\ge r_1=2.
	\end{equation}
	By redefining the step function $A$ on small intervals $[r_{n+1}-\e_n,r_{n+1}]\subset(r_n,r_{n+1}]$ such that its graph coincides with the line segment joining the points $(r_{n+1}-\e_n,r_n^\r)$ and $(r_{n+1},r_{n+1}^\r)$, it remains non-decreasing, becomes continuous and satisfies
	\begin{equation}\label{Eq:example-modified}
	r^\lambda< A(r)\le r^\rho,\quad r\ge r_1=2,
	\end{equation}
	instead of \eqref{Eq:example}, and $A(r_n)=r_n^\rho\le A(r)\le r_{n+1}^\rho$ for all $r\in[r_n,r_{n+1}]$ and $n\in\N$.
	
	Suppose that $\sigma\colon [2,\infty)\to (0,\infty)$ is a quasi proximate order of $A$ such that
	$$
	\limsup_{r\to \infty} \sigma(r)=\rho,  \quad  \liminf_{r\to \infty} \sigma(r)=\lambda.
	$$
	Then there exist infinite sequences $(n_k)_{k\in\N}$ and $(r_{k}')_{k\in\N}$, and a decreasing function $\eta\colon[2,\infty)\to (0,1)$ such that $\lim_{r\to\infty}\eta(r)=0$, $r_{k}'\in(r_{n_k},r_{n_k+1})$ for all $k\in\N$, and
	\begin{equation}\label{Ex:2}
	\left(r_k'\right)^{\sigma(r_k')}\le\left(r_k'\right)^{\lambda+\eta(r_k')},\quad k\in\N.
	\end{equation}
	Denote $$C_k=\sup_{[r_{k}', r_{n_k+1}]}|\sigma'(r)|r\log r$$ for all $k\in\N$. By the property (3) of Theorem~\ref{t:valir} we deduce $r_{n}^\r=A(r_{n})\le r_{n}^{\sigma(r_{n})}$, and thus $\rho\le\sigma(r_n)$ for all $n\in\N$. By (3) we also deduce $r_{n_k}^\r\le A(r_k')\le(r_k')^{\s(r_k')}$ for all $k\in\N$. These observations together with \eqref{Ex:2} and the definition of the sequence $(r_n)_{n\in\N}$ yield
	\begin{equation*}
	\begin{split}
	\rho-(\lambda+\eta(r_{k}'))
	&\le \sigma(r_{n_k+1})-\sigma (r_k')
	\le \int_{r_k'}^{r_{n_k+1}} \frac{C_k}{r\log r} \, dr\\
	&
	=C_k \log \frac{ \log r^\lambda _{n_k+1}}{\lambda \log r_k'}
	=C_k \log \frac{ \log r^\r _{n_k}}{\lambda \log r_k'}\\
	&\le C_k \log \frac{ \log (r' _{k})^{\lambda+\eta(r_k')}}{\lambda \log r_k'}=C_k \log \frac{\lambda+\eta(r_k')}{\lambda}.
	\end{split}
	\end{equation*}
	It follows that
	$$
	C_k\ge\frac{\rho-\lambda-\eta(r_k')}{\log \frac{\lambda+\eta(r_k')}{\lambda}},\quad k\in\N,
	$$
	and, consequently, $\limsup_{r\to\infty}|\sigma'(r)|r\log r=\infty$. 
\end{example}

\section{Estimates for the number of zeros in polar rectangles}\label{sec:polar-rectangle}

Let $\{a_k\}$ denote the sequence of zeros of $f$, listed according to their multiplicities and ordered by the increasing moduli. Denote
	$$
	n_1(r,f)=\max_\varphi \# \left\{a_k:r\le\left|a_k\right|\le\frac{1+r}{2},\,\left|\arg a_k-\varphi\right|\le\frac{\pi}{4}(1-r)\right\},\quad 0\le r<1.
	$$
The aim of this section is to establish the following sharp estimate for $n_1(r,f)$ in terms of a quasi proximate order of $f$. 

\begin{proposition}\label{p:i_est}
Let $f$ be an analytic function in $\D$ such that $0\le \lambda_M(f)<\sigma_M(f)<\infty$. Let $\lambda$ be a quasi proximate order of $f$, related to $\eta\in(0,\sigma_M(f)-\lambda_M(f))$, and $\varepsilon>0$. Then
		\begin{equation}\label{e:n_est0}
		n_1(R, f) \lesssim \frac{1}{(1-R)^{1+\lambda(R)\left(1-\frac{1}{\sigma_M(f)}\right)^+ +\varepsilon}},\quad R\to1^-.
		\end{equation}
In particular,
	\begin{equation}\label{e:n_est}
  n_1(R,f)\lesssim \frac{1}{(1-R)^{1+\left(\lambda_M(f)+\eta\right)\left(1-\frac{1}{\sigma_M(f)}\right)^++\varepsilon}},\quad R\in E,
  \end{equation}
where $E\subset[0,1)$ satisfies $\DU(E)=1$.
\end{proposition}

The first step towards Proposition~\ref{p:i_est} is to apply \cite[Theorem~2]{Linden1956}. Namely, if $n(\zeta,h,f)$ denotes the number of zeros of an analytic function $f$ in the closed disc $\overline{D}(\zeta,h)=\{w:|\zeta-w|\le h\}$, then by replacing $z$ by $Rz$ in the said theorem we obtain the following result.

\begin{lettertheorem}\label{TheoremLinden2}
Let $f$ be an analytic function in $\D$ such that $f(0)=1$. Then, for $\a\in[1/2,1)$ and $\widetilde\eta\in(0,1/6)$ there exist constants $R_0=R_0(\a)\in(0,1)$ and $C=C(\a,\widetilde\eta)$ such that
    \begin{equation}\label{e:n}
    n(\zeta,h,f)\le\frac{C}{(R-r)^\frac1\a}\left(\int_0^R\log^+M(t,f)(R-t)^{\frac1\a-1}\,dt+\log^+M(R_0,f)\right),
    \end{equation}
where $|\zeta|=r<R<1$ and $h=\widetilde\eta(R-r)/R$.
\end{lettertheorem}

In view of this theorem the proof of Proposition~\ref{p:i_est} boils down to estimating the integral appearing on the right-hand side of \eqref{e:n}. This is done in the following lemma.

\begin{lemma}\label{l:i_est}
Let $f$ be an analytic function in $\D$ such that $0\le \lambda_M(f)<\sigma_M(f)<\infty $.
For $1/2\le\alpha<1$ and $0\le R_0<1$, define $I_\alpha:[R_0,1)\to[0,\infty)$ by
  \begin{equation}\label{e:ir_def}
  I_\alpha(R)=\frac{1}{(1-R)^{\frac1\a}}\left(\int_0^R\log^+M(t,f)(R-t)^{\frac1\a-1}\,dt+\log^+M(R_0,f)\right).
  \end{equation}
Let $\lambda$ be a quasi proximate order of $f$ related to $\eta\in(0,\sigma_M(f)-\lambda_M(f))$. Then
	\begin{equation}\label{e:ir_est0}
  I_\alpha(R)\lesssim\frac{1}{(1-R)^{1+\lambda(R)\left(1-\frac{1}{\sigma_M(f)}\right)^+ +\varepsilon}},\quad R\to1^-,
	\end{equation}
provided one of the following two conditions is satisfied:
	\begin{itemize}
	\item[\rm(i)] $0\le\sigma_M(f)<1$ and $\e>0$;
	\item[\rm(ii)] $1\le\sigma_M(f)<\infty$ and $1/\a<1+\frac{\varepsilon}2$.
	\end{itemize}
%
In particular,
	\begin{equation}\label{e:ir_est}
  I_\alpha(R)\lesssim\frac{1}{(1-R)^{1+(\lambda_M(f)+\eta)\left(1-\frac{1}{\sigma_M(f)}\right)^+ +\varepsilon}},\quad R\in E,
  \end{equation}
where $E=E_{\varepsilon,\eta}\subset[0,1)$ satisfies $\DU(E)=1$.
\end{lemma}

\begin{proof}
By Theorem~\ref{t:valir} there exists a quasi proximate order $\lambda$ of $f$, related to $\eta\in(0,\sigma_M(f)-\lambda_M(f))$, satisfying the properties (1)--(7).
In view of (7) the problem boils down to estimating the quantity
		\begin{equation}\label{e:j_def}
		J(R)=\frac{1}{(1-R)^{\frac1\a}}\left(\int_0^R \frac{(R-t)^{\frac1\a-1}}{(1-t)^{\lambda(t)}}\,dt+1\right).
		\end{equation}
		
(i) Assume that $0\le\sigma_M(f)<1$ and, without loss of generality, pick up $\e>0$ such that $\sigma_M(f)+\varepsilon<1/\a$. Fix $R_0\in(0,1)$ such that $\lambda(r)\le \sigma_M(f)+\varepsilon$ for all $r\in[R_0,1)$, and let $R\in(R_0,1)$. Then
    \begin{equation}\label{pervo1}
    J(R)\lesssim\frac 1{(1-R)^{\frac1\a}} \int_0^R \frac{dt}{(1-t)^{\sigma_M(f)+\varepsilon+1-\frac 1\alpha}}
      \lesssim \frac 1{(1-R)^{\frac1\a}} \lesssim \frac 1{(1-R)^{1+\varepsilon}},
    \end{equation}
and thus \eqref{e:ir_est0} is proved.

(ii) Assume next that $1\le\sigma_M(f)<\infty$ and $\varepsilon<1/\a<1+\frac{\varepsilon}2$. Fix $R_0\in(0,1)$ such that $\lambda(r)\le \sigma_M(f)+\frac\varepsilon2$ for all $r\in[R_0,1)$, and let $R\in(R_0,1)$. If $\lambda(R)\ge\sigma_M(f)$ the required estimate follows from standard arguments applied directly to \eqref{e:ir_def}. Assume that $\lambda(R)<\sigma_M(f)$, and define $R^*\in(0,R)$ by the condition $(1-R^*)^{\sigma_M(f)}=(1-R)^{\lambda(R)}$. Then (5), (6), the inequality $1/\a<1+\frac{\varepsilon}2$, and the definition of $R^*$ yield
		\begin{equation*}
  	\begin{split}
		J(R)&=(1-R)^{-\frac1{\alpha}} \left(\int_0^{R^*}+\int_{R^*}^R\right)\frac{(R-t)^{\frac1\a-1}}{(1-t)^{\lambda(t)}}\,dt\\
		&\lesssim(1-R)^{-\frac1{\alpha}} \int_0^{R^*}\frac{(R-t)^{\frac1\alpha-1}}{(1-t)^{\sigma_M(f)+\varepsilon/2}}\,dt
		+(1-R)^{-\frac1{\alpha}} \int_{R^*}^R A^*(t) (R-t)^{\frac1\alpha-1}\,dt\\
		&\le(1-R)^{-\frac1{\alpha}} \int_0^{R^*} \frac1{(1-t)^{\sigma_M(f)+\varepsilon/2-\frac 1\alpha +1}}\, dt+ \alpha (1-R)^{-\frac1{\alpha}} {A^*(R) (R-R^*)^{\frac1\alpha}}\\
		&\lesssim \frac{{(1-R)^{-\frac1{\alpha}}}}{(1-R^*)^{\sigma_M(f)+\varepsilon/2 -\frac 1\alpha}}
		+\frac{(1-R^*)^{\frac1\alpha}}{(1-R)^{\lambda(R)+\frac 1\alpha }}\\
		&=\frac1{(1-R)^{\frac1{\alpha}+\frac{\lambda(R)}{\sigma_M(f)}({\sigma_M(f)+\varepsilon/2 -\frac 1\alpha})   }}
		+\frac{1}{(1-R)^{\lambda(R)+\frac 1	\alpha -\frac{\lambda(R)}{\sigma\alpha} }}\\
		&\lesssim\frac1{(1-R)^{\frac1{\alpha}+\lambda(R)+ \varepsilon/2 -\frac{\lambda(R)}{\sigma_M(f)\alpha}}}
		\lesssim\frac1{(1-R)^{1+\lambda(R)-\frac{\lambda(R)}{\sigma_M(f)}+\varepsilon}}.
		\end{split}
		\end{equation*}
Therefore \eqref{e:ir_est0} holds in this case also.

It remains to prove \eqref{e:ir_est}. In the case $\sigma_M(f)\le1$ this follows with $E=[R_0,1)$ from \eqref{pervo1}. So assume that $\sigma_M(f)>1$. Pick up an increasing sequence $(R_n)_{n\in\N}$ such that $\lambda(R_n)=\lambda_M(f)+\eta+\frac\varepsilon2$ for all $n\in\N$ and $\lim_{n\to\infty}R_n=1$. Then define $(R^*_n)$ by $(1-R^*_n)^{\lambda_M(f)+\eta+2\varepsilon}=(1-R_n)^{\lambda_M(f)+\eta +\frac {3\varepsilon}2}$ for all $n\in\N$, and set $E=\bigcup_{n=1}^\infty[R_n^*,R_n]$. The properties (5) and (6) imply
	$$
	\frac1{(1-R)^{\lambda(R)}}\lesssim A^*(R)\le A^*(R_n)\le\frac1{(1-R_n)^{\lambda(R_n)}},\quad R\in [R_n^*,R_n],\quad n\in\N,
	$$
and hence \eqref{e:ir_est0}, the definitions of $(R_n)$ and $(R^*_n)$ yield
	\begin{equation*}
	\begin{split}
  J(R)&\lesssim\frac1{(1-R)^{1+\lambda(R)\left(1-\frac{1}{\sigma_M(f)}\right)+\varepsilon}}
	=\frac1{(1-R)^{1+\frac \varepsilon{\sigma_M(f)}}}\cdot\frac1{(1-R)^{(\lambda(R)+\varepsilon)\left(1-\frac{1}{\sigma_M(f)}\right)}}\\
  &\lesssim\frac1{(1-R)^{1+\frac\varepsilon{\sigma_M(f)}}}
	\cdot\frac1{(1-R_n)^{(\lambda(R_n)+\varepsilon)\left(1-\frac{1}{\sigma_M(f)}\right)}}\\
	&=\frac1{(1-R)^{1+\frac\varepsilon{\sigma_M(f)}}}\cdot\frac1{(1-R_n)^{(\lambda_M(f)+\eta+\frac{3\varepsilon}2)\left(1-\frac{1}{\sigma_M(f)}\right)}}\\
	&=\frac1{(1-R)^{1+\frac\varepsilon{\sigma_M(f)}}}\cdot\frac1{(1-R_n^*)^{(\lambda_M(f)+\eta+2\varepsilon)\left(1-\frac{1}{\sigma_M(f)}\right)}}\\
	&\le
\frac1{(1-R)^{1+ (\lambda_M(f)+\eta)\left(1-\frac{1}{\sigma_M(f)}\right)+2\varepsilon-\frac{\varepsilon}{\sigma_M(f)}}},\quad R\in [R_n^*,R_n],\quad n\in\N.
	\end{split}
	\end{equation*}
By re-defining $\e$, this yields \eqref{e:ir_est} on $E$. Finally, since $R_n^*\to1^-$ as $n\to\infty$, we have
	$$
	1\ge\DU(E)\ge\lim_{n\to\infty}\frac{R_n-R_n^*}{1-R_n^*}=1-\lim_{n\to\infty}(1-R_n^*)^{\frac{\varepsilon}{2(\lambda+\eta+\frac{3\varepsilon}{2})}}=1.
	$$
This finishes the proof of the lemma.
\end{proof}

With these preparations we can prove the proposition we are after in this section.

\begin{proof}[Proof of Proposition~\ref{p:i_est}] For simplicity, we assume that $f(0)=1$. The general case is an easy modification. 
Let $\varepsilon>0$ be given, and fix $\alpha\in[1/2,1)$ and $\eta\in(0,\sigma_M(f)-\lambda_M(f))$ as in Lemma~\ref{l:i_est}. Further, let $(R_n)_{n\in\N}$ be the sequence appearing in the proof of the said lemma. Let $r_\nu=1-2^{-\nu}$ for all $\nu\in\N$, and let $\nu$ be sufficiently large such that $r_\nu\ge R_0$, where $R_0=R_0(\alpha)\in(0,1)$ is the constant from Theorem~\ref{TheoremLinden2}.
Let $r\in[r_\nu,r_{\nu+1})$ and $R=\frac{2r}{1+r}\in(r_{\nu},r_{\nu+2})$. Then the estimate \eqref{e:n} yields
    \begin{equation}\label{e:nglob}
    n(\zeta,\tilde\eta(1-r)/2,f)
		\lesssim I_\alpha\left(\frac{2r}{1+r}\right),
    \end{equation}
where $|\zeta|=r$ and $0<\widetilde\eta<1/6$. An application of this estimate together with Lemma~\ref{l:i_est} gives the assertions.
\end{proof}

\section{Generalization of Cartwright's lemma and its application}\label{sec:cartright}

In this section we prove Proposition~\ref{cor:cartright} and establish Theorem~\ref{t:cartright} as its consequence. Both results are stated in the introduction. The proposition generalizes the Cartwright's lemma, see Lemma~\ref{Lemma:B}, while the theorem concerns estimates of real parts of analytic functions. To establish our results we will need the following Warschawski mapping theorem. This theorem was efficiently used in the earlier work~\cite{Bor_min} on the topic. 

\begin{lettertheorem}[{\cite[Sec.~7]{War41}}]\label{t:war}
	Let  $\omega:[0,\infty)\to(0,\infty)$ be a continuously differentiable function such that $\int_0^\infty\frac{(\omega'(t))^2}{\omega(t)}\,dt<\infty$ and $\lim_{t\to\infty}\omega'(t)=0$. Let $\zeta$ be a conformal map of the curvilinear strip $\{u+iv:|v|\le\omega(u)\}$ onto the strip ${\mathcal K}=\{z:|\Im z|\le\pi/2\}$ such that $\zeta(u+iv)\to\pm \infty$ as $u\to\pm \infty$. Then there exists a constant $k=k(\omega)\in\mathbb{R}$ such that
	$$
	{\zeta}(u+iv)=k+\frac{\pi}2\int_0^u\frac{dt}{\omega(t)}+
	i\frac{\pi v}
	{2\omega(u)}+o(1), \quad u\to\infty.
	$$
\end{lettertheorem}

Proposition~\ref{cor:cartright} is a consequence of the following generalization of Cartwright's lemma.

\begin{lemma}\label{l:cartwr}
	Let $l:[1,\infty)\to(0,\infty)$ be a quasi proximate order such that $$0<\liminf_{t\to\infty}l(t)=l_1<\limsup_{t\to\infty}l(t) =l_2<\infty.$$ Define $L(r)=\int_1^r \frac{l(s)}{s}\,ds /\log r$ for all $1\le r<\infty$. 
	Let $\varepsilon>0$, $0<q<1$ and $0<R<\infty$. Let $G$ be analytic and non-vanishing such that
	\begin{equation}\label{hypothesis1}
	\log|G(re^{i\theta})|<r^{{L(r)}}
	\end{equation}
on the domain $\left\{re^{i\theta}:R\le r<\infty,\,|\theta|\le\frac{\pi}{2l(r)q}\right\}$. Then, for each fixed $\delta>0$ there exists $K=K(\delta, l, q)>0$ such that  we have
	\begin{equation}\label{e:lower_cart_ets}
	\log|G(re^{i\theta})|
	>-Kr^{L(r)}, \quad |\theta|\le\frac{\pi}{2l(r)q}-\delta,\quad r\to \infty.
	\end{equation}
\end{lemma}

\begin{proof}
	Let $0<q<1$ and define $\omega(t)=\frac{\pi}{2l(e^t)q}$ for all $0\le t<\infty$. Since
	\begin{equation} \label{e:prox_ord_prop}
	\limsup_{t\to\infty}|l'(t)|t\log t<\infty
	\end{equation}
	and $0<\liminf_{t\to\infty}l(t)=l_1<\limsup_{t\to\infty}l(t) =l_2<\infty$ by the hypotheses, we have
	$$
	|\omega'(t)|=\frac{\pi|l'(e^t)|e^t}{2l^2(e^t)}\lesssim\frac1t, \quad t\to \infty.
	$$
	Therefore, for sufficiently large $R>0$ we have
	\begin{equation*}
	\int_R^\infty\frac{(\omega'(t))^2}{\omega(t)}\,dt
	\lesssim\int_R^\infty\frac{dt}{t^2}<\infty.
	\end{equation*}
By re-defining $\om$ on $[0,R)$ in a suitable way if necessary we deduce that it satisfies the hypotheses of Theorem~\ref{t:war}. Since the statement of the lemma concerns the behavior of $G$ far a way from the origin, this re-definition does not affect to the proof.
	
For  $\alpha>0$, define 
	$$
	S_q(\alpha)=\left\{w=\rho e^{i\varphi}: \rho>0,\,|\varphi| \le \frac{\pi } {2\alpha q}\right\},$$ 
	$$	\widetilde S_q=\left\{z=re^{i\theta}:r>0,\,|\theta|\le\frac{\pi} {2l(r)q}\right\},
	$$
	and
	$$
	\log\widetilde S_q=\Bigl\{u+iv: -\infty< u<+\infty,\, |v|\le\om(u)=\frac{\pi}{2l(e^u)q}\Bigr\}.
	$$
	Let now $\zeta$ be a conformal map from $\log\widetilde S_q$ onto $\mathcal{K}$ such that $\zeta(z)\to\pm \infty$ as $z\to\pm \infty$. Consider the functions
	$$
	\psi\colon S_q(\alpha)\to \widetilde S_q,
	\quad z=\psi(w)=\exp \left( \zeta^{-1}(\alpha q\log w)\right),
	$$
	and
	$$
	\psi^{-1}\colon\widetilde S_q \to S_q(\alpha),
	\quad w=\psi^{-1}(z)=\exp\left(\frac1{\alpha q}\zeta(\log z)\right),
	$$
	where $\log$ denotes the principal branch of the logarithm. Define $F=G\circ\psi$ on $S_q(\alpha)$. By the hypothesis \eqref{hypothesis1} we have
	\begin{equation}\label{e:Gupp}
	\log|F(w)|=\log|G(\psi(w))|<|\psi(w)|^{{L(|\psi(w|)}},\quad|\psi(w)|\ge R.
	\end{equation}
	In order to apply Lemma~\ref{Lemma:B} we have to verify \eqref{e:cart_assum} for $F$. In view of \eqref{e:Gupp} it is enough to show $|\psi(w)|^{{L(|\psi(w)|)}}<B|w|^\alpha$, or equivalently,
	\begin{equation}\label{e:dag}
	r^{{L(r)}}\le B|\psi^{-1}(re^{i \theta})|^\alpha.
	\end{equation}
	By Theorem~\ref{t:war} there exists $k\in \mathbb{R}$ such that
	\begin{equation}\label{vehe}
	\begin{split}
	{\zeta}(\log r +i\theta)
	&=k+\frac{\pi}2\int_0^{\log r}\frac{dt}{\omega(t)}+i\frac{\pi \theta}{2\omega(\log r)}+o(1),\quad r\to\infty,\\
	&=k+q \int_0^{\log r} l(e^t)\,dt + i q\theta l(r) +o(1),\quad r\to\infty,\\
	&=k+q\int_1^r \frac{l(s)}s\,ds +iq\theta l(r) +o(1),\quad r\to\infty,
	\end{split}
	\end{equation}
	provided $|\theta|\le\om(\log r)=\frac{\pi}{2l(r)q}$, and consequently,
	\begin{equation*}
	\begin{split}
	|\psi^{-1}(re^{i\theta})|^\alpha
	&=\left| e^{\frac 1q \zeta(\log r+i\theta)}\right|
	=\exp\left(\frac kq +\int_1^r\frac{l(s)}{s}\,ds+o(1)\right)\\
	&=r^{L(r)} e^{\frac kq +o(1)},\quad r\to\infty,
	\end{split}
	\end{equation*}
	whenever $|\theta|\le\frac{\pi}{2l(r)q}$. It follows that $|\psi^{-1}(re^{i\theta})|^\alpha=r^{L(r)} d(r)$, where $d(r)=(1+o(1))e^{\frac kq}$ as $r\to\infty$, and hence \eqref{e:dag} is valid with $B=\sup_{r>R} d(r)$. Therefore Lemma~\ref{Lemma:B} implies that, for a given $\delta_1>0$, there exists a constant $K=K(\delta_1)>0$ such that
	$$
	\log|F(w)|=\log|G(\psi(w))|>-KB                      |w|^\alpha, \quad |\arg w|\le \frac{\pi}{2\alpha q}-\delta_1.
	$$
	Now \eqref{vehe} yields
	$$
	\arg \psi^{-1} (re^{i \theta})
	=\Im\left(\frac1{\alpha q}\zeta(\log r+ i\theta))\right)
	\sim\frac{\theta }{\alpha} l(r), \quad r\to\infty,
	$$
	and therefore, by choosing $\delta_1=\delta_1(\delta)>0$ sufficiently small we deduce
	$$
	\log|G(re^{i\theta})|>-K|\psi^{-1}(re^{i\theta})|^\alpha
	=-(1+o(1))K r^{L(r)} e^{\frac kq} \ge -K_1 r^{L(r)},
	$$
	where $	 |\theta|\le \frac{\pi}{2ql(r)}-\delta$.
	The lemma is now proved.
\end{proof}

What remains to be done in this section is to deduce the results we are after from Lemma~\ref{l:cartwr}.

\begin{proof}[Proof of Proposition~\ref{cor:cartright}]
The proposition is an immediate consequence of Lemma~\ref{l:cartwr} because for a generalized proximate order we have $L(r)\sim l(r)$ as $r\to \infty$. In fact, (4') implies that for each $\e>0$ there exists $R=R(\e,l_1)>1$ such that $|l'(s)s\log s|<\frac{\e l_1}3$ for all $s\ge R$. Therefore, for some constant $C=C(R,l)>0$ we have
	\begin{equation*}
	\begin{split}
	\left|L(r)-l(r)\right|
	&=\frac1{\log r}\left|\int_1^rl'(s)\log s\,ds\right|
	\le\frac1{\log r}\left(\int_1^R|l'(s)|\log s\,ds+\frac{\e l_1}3\int_R^r\frac{ds}{s}\right)\\
	&\le\frac{C}{\log r} +\frac{\e l_1}3
	\le\frac{\e l_1}2 , \quad r\to\infty,
	\end{split}
	\end{equation*}
and consequently, $\left|\frac{L(r)}{l(r)}-1\right|<\e$ as $r\to\infty$. 
\end{proof}

\begin{proof}[Proof of Theorem~\ref{t:cartright}] As in \cite{CartWr33}, the statement of the theorem follows from Proposition~\ref{cor:cartright}. Let $\e>0$ and $0\le\theta<2\pi$ fixed. Choose $\e_1>0$, $0<q<1$ and $0<R<\infty$ such that 
	\begin{equation}\label{pillukarva}
	\sup_{r\ge R}\e_1\left(1+\left(1+\e_1\right)\left(\sigma(r)\right)^{-1}\right)<\e
	\quad\textrm{and}\quad
	\inf_{r\ge R}\sigma(r)(1+\varepsilon_1)q>1.
	\end{equation}
	Then the function
	$$
	G(w)=\exp\left(f\left(e^{i\theta}\left(1-\frac 1w\right)\right)\right)
	$$
	satisfies the hypotheses of Proposition~\ref{cor:cartright} with $l(t)=(1+\varepsilon_1)(\sigma(t)+\e_1)$ and sufficiently large $R$. In fact, by writing $w=1/(1-z)$ we have $|\arg w|\le\frac{\pi}{2l(|w|)q}$ if and only if $|\arg \frac1{1-z}|\le\pi\left(2l\left(\frac1{\left|1-z\right|}\right)q\right)^{-1}$, and hence $|1-z|\asymp1-|z|$ by the right hand inequality in \eqref{pillukarva}. This together with the hypothesis \eqref{Eg:hypothesis-cartright} shows that 
	\begin{equation*}
	\begin{split}
	\log|G(w)|
	&=\Re f\left(e^{i\theta}\left(1-\frac 1w\right)\right)
	=\Re f(e^{i\theta}z)
	\le\frac{1}{(1-|z|)^{\lambda(|z|)}}
	\le\frac{C_1}{|1-z|^{\lambda(|z|)}}\\
	&=C_1|w|^{\sigma\left(\frac1{1-\left|1-\frac1w\right|}\right)}
	\lesssim |w|^{\sigma\left(\left|w\right|\right)},\quad |w|\to\infty,
	\end{split}
	\end{equation*}
	for some constant $C_1=C_1(\sigma,\e_1,q)>0$. The last equality is valid because $\sigma$ is a generalized proximate order by the hypothesis, and hence there exist constants $C_2=C_2(w,\sigma,\e_1,q)\ge1$ and $C_3=C_3(\sigma,\e_1,q)\ge1$ such that $\sup_wC_2<\infty$ and
	\begin{equation*}
	\begin{split}
	\left|\sigma\left(\frac1{1-\left|1-\frac1w\right|}\right)-\sigma\left(\left|w\right|\right)\right|
	&=|\sigma'(C_2|w|)|\left(\frac1{1-\left|1-\frac1w\right|}-|w|\right)\\
	&\le|\sigma'(C_2|w|)|C_3|w|=O\Bigl(\frac 1{\log|w|}\Bigr),\quad |w|\to\infty,
	\end{split}
	\end{equation*}
	provided $|\arg w|\le\frac{\pi}{2l(|w|)q}$. Therefore, by choosing $\delta=\frac{\pi}{4l_2q}$ and using the left hand inequality in \eqref{pillukarva} we deduce
	\begin{equation*}
	\begin{split}
	\Re f(re^{i\theta})
	&=\log\left|G\left(\frac1{1-r}\right)\right|>-\left(\frac1{1-r}\right)^{(1+\varepsilon_1)\left(\sigma\left(\frac1{1-r}\right)+\e_1\right)+o(1)}\\
	&>-\frac{1}{(1-r)^{(1+\varepsilon)\lambda(r)}},\quad r\to1^-.
	\end{split}
	\end{equation*}
	Since $0\le\theta<2\pi$ was arbitrary, the assertion is proved.
\end{proof}

\section{Integral mean estimates and their consequences}\label{sec:final}

In this last section we prove the integral mean estimates stated in Theorem~\ref{t:p_lower_order} and deduce Corollary~\ref{c:m_p} as its consequence. Several auxiliary results are needed. We begin with two known ones concerning the canonical product
	\begin{equation}\label{e:canon_prod}
	\mathcal{P}(z)
	=\mathcal{P}\left(z,(a_k)_{k=1}^\infty,s\right)
	=\prod_{k=1}^\infty E(A(z, a_k), s),\quad \sum_{k=1}^\infty(1-|a_k|)^{s+1}<\infty,
	\end{equation}
where  
	$$
	E(w,s)=(1-w)\exp\left(w+\frac{w^2}{2} +\dots+ \frac{w^s}{s}\right),\quad s\in \N,
	$$ 
is the Weierstrass primary factor and $A(z,\zeta)=\dfrac{1-|\zeta|^2}{1-z\overline\zeta}$ for all $z\in \mathbb{D}$ and $\zeta\in\overline{\mathbb{D}}$. Both of them concern sharp estimates of $\log|\mathcal{P}|$ in terms of sums of $|A(z,a_k)|$'s.

\begin{lettertheorem}[{\cite[p.224]{Tsuji}}]\label{t:tsuji} For each canonical product $\mathcal{P}$ and $\varepsilon>0$ there exists a constant $K=K(\mathcal{P},\e)>0$ such that
	\begin{equation}\label{2.4l}
  \log|\mathcal{P}(z)|\le K \sum_{k=1}^\infty|A(z,a_k)|^{\mu+1+\varepsilon},\quad\frac12\le|z|<1,
  \end{equation} 
where $\mu$ is the exponent of convergence of the zero sequence $(a_k)_{k=1}^\infty$.
Further, if $D_k=D(a_k,(1-|a_k|^2)^{\mu+4})$ for all $k\in\N$, then
	\begin{equation}\label{2.5l}
  \log|\mathcal{P}(z)|\ge K\log(1-|z|)\sum_{k=1}^\infty|A(z,a_k)|^{\mu+1+\varepsilon},
	\quad\frac12\le|z|<1,\quad z\not\in\bigcup_{k=1}^\infty D_k.
  \end{equation}
\end{lettertheorem}
						
It is  known~\cite{DrSh} that P\'olya's order $\rho^*(\psi)$ for a non-decreasing function $\psi:[r_0,\infty)\to(0,\infty)$ can be determined by the formula
\begin{equation}\label{e:psi_con}
\rho^*(\psi)
=\sup\left\{\rho:\limsup_{x,C	\to\infty}\frac{\psi(Cx)}{C^\rho\psi(x)}=\infty\right\}.
\end{equation}
It is finite if and only if $\psi$ satisfies Caramata's condition $\psi(2t)\lesssim\psi(t)$ as $t\to\infty$. The P\'olya's order $\rho^*(\psi)$ is not smaller than the order of growth of $\psi$, that is, $\psi(x)\lesssim x^\rho$ as $x\to\infty$ for all $\rho>\rho^*(\psi)$, but not vice versa. 
						
\begin{letterlemma}[{\cite[Lemma 9]{ChyShep}}]\label{l:lin_can} Let $f$ be an analytic function in $\D$ with a zero sequence $Z=(a_k)_{k=1}^\infty$ such that
	$$
	n_1\left(r,f\right)\le\psi\left(\frac1{1-r}\right),\quad 0\le r<1, 
	$$
where the function $\psi$ satisfies Caramata's condition $\psi(2t)\lesssim\psi(t)$ as $t\to\infty$. Then, for $s>\rho^*(\psi)$, there exists a constant $C>0$ such that the canonical product $\mathcal{P}(z)=\mathcal{P}\left(z,\left(a_k\right)_{k=1}^\infty,s\right)$ admits the estimate
	\begin{equation}\label{e:lem1}
	\log|\mathcal{P}(z)|\le2^{s+2}\sum_{k=1}^\infty \left|A(z,a_k)\right|^{s+1} 
	\le C\widetilde\psi\left(\frac1{1-|z|}\right), \quad z\in\D,
	\end{equation}
where $\widetilde{\psi}(t)=\int_{1}^t \frac{\psi(x)}{x}\,dx$.
\end{letterlemma}

The main step towards Theorem~\ref{t:p_lower_order} and Corollary~\ref{c:m_p} is the following lemma which is a counterpart of \cite[Lemma 1]{Li_mp} to the language used in this paper. Its proof relies on Proposition~\ref{p:i_est} and Lemma~\ref{l:i_est} as well as on a number of auxiliary results from the existing literature of which two are stated above.
     
\begin{lemma}\label{l:lin1}
Let $f$ be analytic in $\D$ such that $0\le\lambda_M(f)<\sigma_M(f)<\infty$, and $1\le p<\infty$. Let $\varepsilon>0$ and let $\lambda$ be a quasi proximate order of $f$, related to $\eta\in (0,\sigma_M(f)-\lambda_M(f))$. Then
	$$
	m_p\left(r,\log\left|\mathcal{P}\left(\cdot,(a_k)_{k=1}^\infty,s\right)\right|\right)
	\lesssim\frac{1}{(1-r)^{\lambda(r)\left(1-\frac1{\sigma_M(f)}\right)^++1+\varepsilon}},\quad 0\le r<1,
	$$
where $(a_k)_{k=1}^\infty$ is the zero sequence of $f$, $s=[\rho^*(A^*)]+1$ and $A^*$ is the function from the definition of the quasi proximate order $\lambda$.
\end{lemma}

\begin{proof}
If $\sigma_M(f)\le1$, then 
the assertion is an immediate consequence of \cite[Lemma~1]{Li_mp}. Assume now that $\sigma_M(f)>1$.
Define 
	\begin{equation}\label{Eq:def-psi}
	\psi(t)=\left(A^*\left(1-\frac 1t\right)\right)^{1-\frac1{\sigma_M(f)}}t^{1+\varepsilon},\quad 1\le t<\infty,
	\end{equation}
where $A^*$ is the function from the definition of $\lambda$. By Proposition~\ref{p:i_est} and property (5) we have
	\begin{equation}\label{e:zero_est}
	n_1(r,f)
	\lesssim\left(\frac{1}{(1-r)^{\lambda(r)}}\right)^{1-\frac 1{\sigma_M(f)}}\frac{1}{(1-r)^{1+\varepsilon}}
	\lesssim\frac{(A^*(r))^{1-\frac1{\sigma_M(f)}}}{(1-r)^{1+\varepsilon}}
	=\psi\left(\frac1{1-r}\right).
	\end{equation} 
By the property (6) of a quasi proximate order, $\psi$ satisfies Caramata's condition, and consequently, $\psi$ has finite P\'olya order
$\rho^*=\rho^*(\psi)$. Therefore we may apply Lemma~\ref{l:lin_can} with $s=[\rho^*(A^*)]+1$. Moreover, as $A^*$ is nondecreasing by (6), we have
	\begin{equation}\label{e:zero_tild_est}
	\begin{split}
	\widetilde{\psi}(t)
	&=\int_1^t \frac{\psi(x)}{x} \, dx=
	\int_0^t\left(A^*\left(1-\frac1x\right)\right)^{1-\frac 1{\sigma_M(f)}}x^{\varepsilon}\,dx\\
	&\le\frac1{1+\varepsilon}\left(A^*\left(1-\frac1t\right)\right)^{1-\frac 1{\sigma_M(f)}}t^{1+\varepsilon}
	\le\psi(t),\quad 1\le t<\infty.
	\end{split}
	\end{equation}

To establish the statement we follow the scheme of the proof of \cite[Lemma~1]{Li_mp}. Without loss of generality, we may assume that $1/2\le r<1$ and $\frac 34 \le |a_k|<1$ for all $k\in\N$. We cover the set of integration by the intervals $[\tau+r-1,\tau+1-r]$, where $\tau=2k(1-r)$ and $k=0,1,\dots, [\frac{\pi}{1-r}]$, and show that
	\begin{equation}\label{2.9l}
  \int_{\tau+r-1}^{\tau+1-r}|\log|\mathcal{P}(re^{i\theta})||^p\,d\theta 
	\lesssim\frac{(1-r)\left(\log\frac1{1-r}\right)^p}{(1-r)^{\left(\lambda(r)\left(1-\frac 1{\sigma_M(f)}\right)+1+\varepsilon\right)p}}
  \end{equation}
for each $\tau$. Then the statement of the lemma follows by summing over $\tau$. 

Without loss of generality we may suppose that $\tau=0$. For given $r$, let $F$ denote the set of integers $m$ for which the exceptional disc $D_m$ of Theorem \ref{t:tsuji} intersects $\gamma_r=\{ re^{i\theta}: r-1\le \theta \le 1-r\}$. 

Application of \eqref{e:nglob} and Lemma~\ref{l:i_est} show that
	\begin{equation}\label{2.11l}
  \# F\lesssim \psi\left(\frac 1{1-r}\right).
  \end{equation}
		
Consider the factorization $\mathcal{P}=B_1B_2B_3$, where
	$$
  B_1(z)=\prod_{k\not\in F} E(A(z,a_k),s), \quad B_2(z)=\prod_{k\in F}\exp\left(\sum_{j=1}^s\frac1j(A(z,a_k))^j\right),  
	$$
and
	$$
	B_3(z)=\prod_{k\in F}(1-A(z,a_k))=\prod_{k\in F}\frac{\overline a_k(a_k-z)}{1-z\overline a_k}.
	$$
It follows from \eqref{e:lem1}  that  the series  $\sum_{k=1}^\infty \Bigl( \frac{1-|a_k|^2}{|1-z\bar a_k|}\Bigr)^{s+1} $   is convergent for every $z\in \D$ provided that $s>\rho^*(\psi)$. Choosing $z=0$ we deduce that the convergence exponent $\mu$ of the sequence $(a_k)_{k=1}^\infty $ satisfies   $\mu \le \rho^*(\psi)$. Thus,  using  Theorem~\ref{t:tsuji} with $\varepsilon=s-\mu$ we obtain
	\begin{equation*}
  \int_{r-1}^{1-r} |\log|{B_1}(re^{i\theta})||^pd\theta
	\lesssim\left(\log\frac1{1-r}\right)^p\int_{r-1}^{1-r}\left(\sum_{k\not\in F}|A(re^{i\theta},a_k)|^{s+1}\right)^p\,d\theta.
  \end{equation*}
Then Lemma~\ref{l:lin_can}, \eqref{e:zero_tild_est} and property (5) yield
	\begin{equation}\label{e:appl_lema9}
	\begin{split}
	\int_{r-1}^{1-r}|\log |{B_1}(re^{i\theta})||^pd\theta 
	&\lesssim\left(\log\frac1{1-r}\right)^p\widetilde\psi\left(\frac1{1-r}\right)^p(1-r)\\
	&\lesssim\frac{(1-r)\left(\log\frac1{1-r}\right)^p}{(1-r)^{\left(\lambda(r)\left(1-\frac1{\sigma_M(f)}\right)+1+\varepsilon\right)p}}.
  \end{split}
	\end{equation}
On the way to \cite[Lemma~1]{Li_mp} it is proved that
  \begin{equation*}
  \int_{r-1}^{1-r} |\log |{B_2}(re^{i\theta})||^pd\theta +\int_{r-1}^{1-r} |\log |{B_3}(re^{i\theta})||^pd\theta  \lesssim (1-r) (\# F)^p.
	\end{equation*}
Taking into account \eqref{2.11l} we deduce
	\begin{equation}\label{e:b2_b3}
  \int_{r-1}^{1-r}|\log |{B_2}(re^{i\theta})||^pd\theta
	+\int_{r-1}^{1-r}|\log |{B_3}(re^{i\theta})||^pd\theta
	\lesssim (1-r)\psi\left(\frac 1{1-r}\right)^p.
	\end{equation}
Estimates \eqref{e:appl_lema9} and \eqref{e:b2_b3} now yield \eqref{2.9l}.
\end{proof}

We can now finish this section and the paper by proving Theorem~\ref{t:p_lower_order} and Corollary~\ref{c:m_p}.

\begin{proof}[Proof of Theorem~\ref{t:p_lower_order}]
The statement \eqref{e:mp_final-countdown} is valid by standard estimates if $\sigma_M(f)\le1$. Let now $f$ be an analytic function in $\D$ such that $1<\lambda_M(f)<\sigma_M(f)<\infty$, and let $\lambda$ be its generalized proximate order. Each generalized proximate order is a quasi proximate order. As an associated function one can take $A^*(t)=(1-t)^{-\lambda(t)}$ which is increasing for $r$ sufficiently close to $1$ because of (4'). Thus, the properties (1)--(7) of a quasi proximate order of the analytic function $f$ are satisfied. Further, let $s=[\rho^*(A^*)]+1$, where $\rho^*(A^*)$ is P\'olya's order of $A^*$. Consider the canonical product \eqref{e:canon_prod}, convergent by \cite[Lemma~9]{ChyShep}. It leads us to the factorization
\begin{equation}\label{e:fact}
  f(z)=z^m \mathcal{P}(z)g(z),\quad z\in\D,
\end{equation}
where $g$ is a nonvanishing analytic function in $\D$ and $m=m(f)\in\N\cup\{0\}$. By combining Theorem~\ref{t:tsuji} with Lemma~\ref{l:lin_can} we deduce that there exists $n_0\in\N$ such that, for $n\ge n_0$, the interval $[r_n,r_{n+1})$, where $r_n=1-2^{-n}$, contains a point $T_n$ for which the circle $\{ z:|z|=T_n\}$ does not intersect any of exceptional discs of Theorem~\ref{t:tsuji}, and
	\begin{equation}\label{e:prod_unif_est}
  |\log|\mathcal{P}(z)||
	\lesssim\psi\left(\frac1{1-|z|}\right)\log\frac1{1-|z|},\quad|z|=T_n,
	\end{equation}
where $\psi$ is defined by \eqref{Eq:def-psi}. By using factorization \eqref{e:fact} and the properties (5) and (7), we obtain
	$$
	\log|g(z)|
	\le\log|f(z)|+|\log|\mathcal{P}(z)||-m\log|z|
	\lesssim\psi\left(\frac1{1-|z|}\right)\log\frac1{1-|z|},\quad|z|=T_n.
	$$
The maximum modulus principle together with \eqref{Eq:def-psi} and the properties (5) and (6) then implies
	\begin{equation*}
	\begin{split}
  \log M(r,g)
	&\le\log M (T_{n+1}, f)
	\le\psi\left(\frac1{1-T_{n+1}}\right)\log\frac1{1-T_{n+1}}\\
	&\lesssim\frac{1}{(1-r)^{\lambda(r)\left(1-\frac1{\sigma_M(f)}\right)+ 1+2\varepsilon}},\quad r\in [r_n, r_{n+1}),\quad n\to\infty,
	\end{split}
	\end{equation*}
from which Theorem~\ref{t:cartright} yields
	\begin{equation*}
  |\log|g(z)||
	\lesssim\frac{1}{(1-r)^{(1+\e)\left(\lambda(r)\left(1-\frac1{\sigma_M(f)}\right)+ 1+2\varepsilon\right)}}, \quad r_{n_0}\le |z|<1.
	\end{equation*}
Therefore
	$$
	m_p(r,\log|g|)
	\lesssim\frac{1}{(1-r)^{(1+\e)\left(\lambda(r)\left(1-\frac1{\sigma_M(f)}\right)+ 1+2\varepsilon\right)}},\quad r\to1^-,
	$$
and $m_p(r,\log|f|)$ obeys the same upper bound by Minkowski's inequality and Lemma~\ref{l:lin1}. The statement \eqref{e:mp_final-countdown} of Theorem~\ref{t:p_lower_order} in the case $\sigma_M(f)>1$ follows by re-defining $\e$.

By arguing as in the last part of the proof of Lemma~\ref{l:i_est} we deduce \eqref{e:mp_final-countdown2} from \eqref{e:mp_final-countdown}. Details are omitted.
\end{proof}

\begin{proof}[Proof of Corollary~\ref{c:m_p}]
It follows from \eqref{e:mp_final-countdown2} that 
	$$
	\lambda_p(f)\le\left(\lambda_M(f)+\eta\right)\left(1-\frac1{\sigma_M(f)}\right)^++1+\varepsilon.
	$$
Since $\eta$ and $\varepsilon$ can be chosen arbitrarily small, the first inequality follows. The second inequality is a consequence of the estimate
	$$ 
	\log|f(re^{i\theta})|\lesssim \frac{m_p(\frac{1+r}{2}, \log|f|)}{(1-r)^{\frac 1p}},\quad r\to1^-,
	$$
which follows from the Poisson-Jensen formula and H\"older's inequality~\cite{Li_mp}.
\end{proof}


\begin{thebibliography}{2}

\bibitem{Boichuk} 		V.~S.~Boichuk, 
											A class of entire functions (Russian), 
											Ukrain. Mat. Zh. 38 (1986), no. 6, 683--688, 813. 

\bibitem{Bor_min} 		A.~Borichev, 
											On the minimum of harmonic functions, 
											J. Anal. Math. 89 (2003), 199--212. 

\bibitem{CartWr33} 		M.~L.~Cartwright, 
											On analytic functions regular in the unit circle (I),  
											Quart. J. Math. 4 (1933), 246--257.

\bibitem{Chernolyas}	V.~I.\v{C}ernoljas,
											The distribution of the roots of entire functions with a generalized 
											completely regular growth of oscillating order (Russian), 
											Dokl. Akad. Nauk SSSR 224 (1975), no. 3, 549--552. 

\bibitem{rho_infty}		I.~Chyzhykov,
											Zero distribution and factorization of analytic functions of slow growth in the unit disc, 
											Proc. Amer. Math. Soc. 141 (2013), no. 4, 1297--1311. 

\bibitem{ChIJM}				I.~Chyzhykov,
											Asymptotic behaviour of $p$th means of analytic and subharmonic functions in the unit disc 
											and angular distribution of zeros, 
											Israel J. Math. 236 (2020), no. 2, 931--957.

\bibitem{CHR2010}			I.~Chyzhykov, J.~Heittokangas and J.~R\"atty\"a,
											Sharp logarithmic derivative estimates with applications to ordinary differential equations in the unit disc,
											J. Aust. Math. Soc. 88 (2010), no.~2, 145--167.

\bibitem{ChyShep} 		I.~Chyzhykov and I.~Sheparovych, 
											Interpolation of analytic functions of moderate growth in the unit disc and 
											zeros of solutions of a linear differential equation, 
											J. Math. Anal. Appl. 414 (2014), no. 1, 319--333. 

\bibitem{DrSh} 				D. Drasin and D. Shea, 
											P\'olya peaks and the oscillation of positive functions, 
											Proc. Amer. Math. Soc. 34 (1972), 403--411.

\bibitem{Drasin74} 		D. Drasin, 
											A flexible proximate order, 
											Bull. London. Math. Soc. 6 (1974), 129--135.

\bibitem{GO} 					A.~A.~Goldberg and I. V.~Ostrovskii, 
											Value distribution of meromorphic functions. 
											Translated from the 1970 Russian original by Mikhail Ostrovskii. 
											With an appendix by Alexandre Eremenko and James K. Langley. 
											Translations of Mathematical Monographs, 236. American Mathematical Society, Providence, RI, 2008. 

\bibitem{JK} 					O.~P.~Juneja and G.~P.~Kapoor, 
											Analytic Functions -- Growth Aspects,  
											Pitman Publishing inc. Boston-London-Melbourne, 1985.

\bibitem{Levin}        	B.~Ja.~Levin, 
												Distribution of zeros of entire functions. 
												Translated from the Russian by R. P. Boas, J. M. Danskin, F. M. Goodspeed, 
												J. Korevaar, A. L. Shields and H. P. Thielman. Revised edition. 
												Translations of Mathematical Monographs, 5. American Mathematical Society, Providence, R.I., 1980.

\bibitem{Li_mp} 			C.~N.~Linden,
											Integral logarithmic means for regular functions, 
											Pacific J. Math. 138 (1989), no. 1, 119--127.

\bibitem{Li56_conj}		C.~N.~Linden,
											Functions regular in the unit circle, 
											Proc. Cambridge Philos. Soc. 52 (1956), 49--60.

\bibitem{Linden1956}  C.~N.~Linden, 
											The minimum modulus of functions regular and of finite order in the unit circle, 
											Quart. J. Math. Oxford Ser. (2) 7 (1956), 196--216.

\bibitem{MalKozSad} 	K.~G.~Malyutin, I.~I.~Kozlova and N.~M.~Sadik, 
											Canonical functions of admissible measures in the half-plane (English summary),
											Translation of Mat. Zametki 96 (2014), no. 3, 418--431. Math. Notes 96 (2014), no. 3-4, 391--402.

\bibitem{NikNk74} 			N.~K.~Nikolskiĭ, 
												Selected problems of weighted approximation and spectral analysis. 
												Translated from the Russian by F. A. Cezus. 
												Proceedings of the Steklov Institute of Mathematics, No. 120 (1974).
												American Mathematical Society, Providence, R.I., 1976.

\bibitem{ShiWil}   			A.~L.~Shields and D.~L.~Williams, 
												Bounded projections and the growth of harmonic conjugates in the unit disc,
												Michigan Math. J. 29 (1982), no. 1, 3--25.

\bibitem{Sons68}				L. Sons,
												Regularity of growth and gaps,
												J. Math. Anal. Appl. 24 (1968), 296--306; Corrigendum, J. Math. Anal. Appl. 58 (1977), 232.

\bibitem{Tsuji} 			M.~Tsuji,
											Potential theory in modern function theory. 
											Reprinting of the 1959 original. Chelsea Publishing Co., New York, 1975. 

\bibitem{War41} 				S.~E.~Warschawski,
												On conformal mapping of infinite strips, 
												Trans. Amer. Math. Soc. 51 (1942), 280--335.
\end{thebibliography}
\end{document}